\documentclass[a4paper]{article}

\usepackage[T2A]{fontenc}
\usepackage[cp1251]{inputenc}
\usepackage{amsmath,amssymb,amsthm,amsfonts,array}
\usepackage{hyperref}

\def\F{\mathfrak{F}} 
\def\x{\mathbf{x}}
\def\X{\mathbf{X}}
\def\R{\mathbb R} 
\def\N{\mathbb N} 
\newtheorem{lem}{Lemma}

\begin{document}

\title{Leibniz Differential and Non-Stantard Calculus.}
\author{E.V. Shchepin\\ Mathematical Institute, 119991, Gubkina 8,
Moscow \\E-mail: scepin\@ mi.ras.ru. .}
\maketitle

\begin{abstract}
The designation of the integral of a function by a measure usually contains a differential symbol,
which is not given an independent meaning.
In this article, we will show how to naturally define the concept of a measure differential in both standard and non-standard analysis.
In the standard analysis, the differential-based scheme for determining the integral is a generalization of the Perron integral.
In the non-standard analysis it represents a variation of Loebe's construction of the Lebesgue integral.
\end{abstract}
\bigskip
\bigskip/


\paragraph{Introduction.}
Method of analysis of the infinitesimals lies in the fact that the comparison between
for different objects (for example,  geometric shapes or processes that occur in time), the latter are divided into infinitesimal parts. The corresponding infinitesimal parts are compared
between each other and if the results of the comparisons are the same for all parts,
they are integrated into a similar conclusion for entire objects. The efficiency of the analysis of infinitesimals is due to the fact that when comparing infinitesimals, we can neglect infinitesimals of higher orders.

Let us be given something whole (usually a continuum) $X$ divided into an infinite number of infinitesimal parts called \emph{monads} and some measure $\mu$ defined for parts of this whole. The values of this measure on the  monad $x$ are denoted by $d\mu(x)$ and are called  the measure differentials.
If the measure is additive, then for any finite partitions, the sum of the measures of parts is equal to the measure of the whole. According to Leibniz's "law of continuity" \cite{Kaz}, this extends to infinite partitions and leads to the following basic postulate of the Leibniz integral theory:
\begin{equation}\label{postulate main}
  \int_{X} d\mu(x)=\mu(X)
\end{equation}
 Where $\int_{X}$ denotes Leibniz  "summa omnia"  of all the differentials.

The principle of comparison based on division goes back to Aristotle himself.
For integrals, it is formulated as follows

\begin{equation}\label{postulate Arist}
 \text { if   $\Phi(x)\ge \Psi(x)$   for all monades, then } \int_{X}\Phi(x)\ge \int_{\X}\Psi(x).
\end{equation}

The effectiveness of this theory is due to the following \emph{rule of neglect},
which directly follows from the postulates of Leibniz and Aristotle.
\begin{equation}\label{}
 \text {if  $\Phi(x)$ is infinitesimal for all $x$ then } \int_{X}\Phi(x)d\mu(x)=0
\end{equation}
the last rule could be called the \emph{Berkeley principle}, in honor of Berkeley  famous for his criticism
 of analysis just about the rules of neglect.

This article proposes two different constructions of the Leibniz integral theory, that is, the integral theory based on the concept of an infinitesimal differential and satisfying the above postulates. We will call the approach that treats infinitesimal Cauchy sequences as zero-tending sequences "standard". For functions of one variable, it was implemented by the author in the article \cite{Scepin}.
 The other, "non-standard" treating infinitesimals according to Robinson, is essentially implemented by Loebe \cite{Loeb}.

The reader is not required to be familiar with non-standard analysis.
All we need to expose the non-standard analysis approach to the differential is the concept of an ultra-product.

\paragraph{Monotone partitions and monads.}
Below we will deal with a compact topological space $X$ fixed once and for all.
A partition of a topological space is a cover of this space by closed sets with disjoint interiors.
Also invariably we will consider the sequence of finite partitions  $\{\omega_n\}$ of the space $X$.
However, we will not consider the properties of this sequence unchanged.
The union $\cup\omega_n$ will be denoted $\Omega$. The elements of $\Omega$ will be called \emph{fractions}.

The sequence  $\{\omega_n\}$ is called \emph{monotone} if for any $n$ the partition $\omega_{n+1}$ refines $\omega_{n}$,
i.e. for each fraction $U\in\omega_{n+1}$ there is a unique fraction $\pi(U)\in\omega_n$, such that $U\subset \pi(U)$

 This defines the mapping $\pi\colon\omega_{n+1}\to \omega_n$  called the \emph{inclusion projection}.
Any monotone partition sequence $\{\omega_n\}$ with inclusion projection forms an inverse system.
The inverse limit of the system is denoted simply $\lim \Omega$. The elements of $\lim\omega_n$ will be
called \emph{monotone monads} of the sequence $\{\omega_n\}$.
So for a sequence $\{U_n\in\omega_n\}$ being monotone monad is provided by inclusions $U_{n+1}\subset U_n$ for all $n\in\N$.
A monotone monad $\{U_n\}$ is called \emph{infinitesimal} if the intersection of its members $\cap U_n$ is just one point of $X$.
a monotone partition is called an \emph{infinitesimal} partition if all its monads are such.

\paragraph{Ultra-partitions and ultra-monads.}
Let us fix once and for all a free ultrafilter $\F$ on the set $\N$ of natural numbers.
Recall that a family of subsets $\F\subset 2^\N$  of the  natural series $\N$ is called a free ultrafilter if it satisfies
the following conditions
\begin{enumerate}
  \item $\emptyset\notin\F$, $\N\in\F$
  \item $A,B\in\F$ implies $A\cap B\in\F$
  \item $A\cup B\in\F$ implies $A\in\F$ or $B\in\F$
  \item $\cap \{A\in\F\}=\emptyset$
\end{enumerate}

Two sequences  $\{O_n\}$ and $\{O_n'\}$  of any objects
are called $\F$\emph{-equivalent} if
\begin{equation}\label{hyper-equiv}
  \{n\in\N\mid O_n=O'_n\} \in \F
\end{equation}
In general, a class of $\F$-equivalent sequences is called an \emph{ultra-sequence}.
For a sequence of fraction $\{U_n\}$ the corresponding ultra-sequence  is called \emph{ultra-monad}.
The set of all ultra-monads of the sequence of partitions $\{\omega_n\}$ is called \emph{ ultra-partition} and  denoted by  $*\lim\Omega$.

In other words an ultra-partition is an ultra-product of a sequence of partitions.
A sequencе $\{U_n\}$ is called \emph{ultra-infinitesimal} if it converges to a point with respect to the ultrafilter.

An ultra-partitions  is called \emph{ultra-infinitesimal} if  all its ultra-monads  are ultra-infinitesimals.

\paragraph{Monadic distributions.}
Below we develop simultaneously two theories: "standard" and "nonstandard".
For the sake of brevity we apply the following convention.
\begin{itemize}
\item Definitions and statements with prefix "ultra" in brackets are interpreted in two ways: the first one ignores the contents of the brackets, and the second one opens the brackets.
 \item  The use of ambiguous terms and symbols in a sentence that have a standard and non-standard meaning (for example, "monad" can be a "monotone monad" or "ultra-monad") means that the sentence has two versions: standard and non-standard
 \end{itemize}

Let's now denote monads in bold lowercase Latin letters, and denote the entire set of monads
in bold $\mathbf{X}$ regardless of the monads in question

  Function $\Phi$  is called  \emph{monadic distribution}
on  infinitesimal  partition $\X$   if it match to
every monad $\mathbf{x}\in\mathbf{X}$ of the  partition
a number (ultra-)sequence.

The main example of a monadic distribution is the \emph{measure differential}, whose value on the monad is determined by the formula
\begin{equation}\label{measure-diff}
  d\mu(\{U_n\})=\{\mu(U_n)\},
\end{equation}
where $\mu$ any (possibly non-additive) numeric function on $\Omega$.

\paragraph{The principle of comparison.}

Infinitesimals we deal in integrating, are represented by numerical sequences.
Comparison between two numeric sequences, as well as other operations between sequences is performed
termwise, but unlike addition or multiplication, it may not give any definite result.
If the inequality $x_n\ge y_n$ is satisfied for all members of sequences without exception, then we say that the first sequence \emph{totally} majorates the second and write the usual sign of inequality. If the inequality  holds for all but finitely many terms of the sequences then we say (following \cite{Scepin}), that the first sequence \emph{ eventually} majorates the second.
If the inequality holds only at some element of the ultra-filter $\F$, then we say, that the first sequence \emph{ ultra-}majorates the second.
The eventual majorating  is indicated by the sign $\succ$. and ultra-majorating by the  sign $*\succ$.

The following Lemma, for ultra-majorating, is a particular case of the famous "Transfer Principle"\ ,
and for eventual majorating is the analog of the corresponding Lemma on comparing of differentials
of functions  from\cite{Scepin}.
\begin{lem}
\label{CP}   Let $\mu_1$ and $\mu_2$ be two finitely additive measures on the  $\Omega$,
  such that  $\mu_1\{U_n\}*\succ\mu_2\{U_n\}$ for all  $\{U_n\}\in \prod_{n=1}^{\infty}\omega_n$.  Then $\mu_1 (X)\ge\mu_2(X)$.
\end{lem}
\begin{proof}
  Suppose in contrary $\mu_1 (X)<\mu_2(X)$.
Then for any $n$ one has inequality
  \begin{equation}\label{sum-mu}
      \mu_1 X=\sum_{U\in\omega_n} \mu_1(U)< \sum_{U\in\omega_n} =\mu_2(U)
  \end{equation}
As two above sums contains the same number of summands, there is at least one of
them, says $U_n$ such that
\begin{equation}\label{summand-mu}
 \mu_1(U_n)< \mu_2(U_n)
\end{equation}
In this case one gets the total inequality
$\mu_1 \{U_n\}< \mu_2 \{U_n\}$ in contradiction with our hypothesis for opposite inequality for all ultra-monad.
\end{proof}
The standard (monotone) variant of the comparison principle is proved
just as in \cite{Scepin}.

\paragraph{Leibniz integral.}

Now we are ready to introduce the key concept of the Leibniz integral for monadic distributions.
The integral of the differential of a finite additive measure is determined by the Leibniz formula \eqref{postulate main}. In this case, the
Lemma  \ref{CP} ensures that the Aristotelian comparison principle is fulfilled for integrals of this kind.
The integration of other distributions are made on the basis of their comparison with differentials of measures. Namely
the equality
\begin{equation}\label{int-def}
  \int_{\mathbf{X}} \Phi(\x)=I
\end{equation}
means that for any positive number $\epsilon$, there will be finitely additive measures $\mu_1$,$\mu_2$
  on $\Omega$ for which inequalities are met:
 \begin{enumerate}
   \item $d\mu_2(\x) (*)\succ \Phi(\x)(*)\succ d\mu_1(\x)$ for all $\x\in\omega_\X$
   \item $\mu_2(X)\le I+\epsilon$
   \item $\mu_1(X)\ge I-\epsilon$
 \end{enumerate}
The principle of comparison, for this definition, is obviously fulfilled. So we got two theories for the Leibniz integral at once: standard and non-standard.

The linearity of the Leibniz integral is easily deduced from its definition in the usual way, as is done for example in the article \cite{Scepin}.

As a nontrivial example of the distribution, let's consider the famous Dirac Delta function.

\begin{equation}\label{delta-func}
  \delta_{\mathbf{x}_0}(\mathbf{x})=\left\{ \begin{array}{ll}
                                                \frac{1}{d\mu(\mathbf{x}_0)}, & \hbox{$\mathbf{x}=\mathbf{x}_0$;} \\
                                                0, & \hbox{$\mathbf{x}\ne\mathbf{x}_0$.}
                                              \end{array}
                                            \right.
.
\end{equation}
Integrable monadic distributions generate Schwartz distributions, that is, generalized functions. But unlike generalized functions, they can also be multiplied.
However, when multiplying, you may lose integrability.

\paragraph{Leibniz integral and Lebesgue-Stieltjes integral.}
The most common example of a monadic distribution is the product of the differential of a measure $\mu$ by the value of a function $f(x)d\mu(\x)$. There $\x\in\X$ denotes a (ultra-)monad, and
$x=\lim\x$ denotes its (ultra-)limit point.

The indicator $\mathbf{1}_Y(\mathbf{x})$ of a subset $Y\subset X$ is defined as a Boolean function
on the set of monads $\mathbf{X}$ as follows:
\begin{equation}\label{chi-def}
  \mathbf{1}_Y(\mathbf{x})=\left\{
    \begin{array}{ll}
     1 , & \hbox{$x\in Y$;} \\
     0 , & \hbox{$x\notin Y$.}
    \end{array}
  \right.
\end{equation}
Where $x=\lim\mathbf{x}$.

The following lemma establishes a connection between the  Leibniz integral and
the Lebesgue measure.
\begin{lem}
  \label{Reznichenko} Let $\mu$ be a regular Radon measure  $X$.
  Then for any measurable subset $Y\subset \X$ one has:
  \begin{equation}\label{Rez}
    \int_{\mathbf{X}}\mathbf{1}_Y(\mathbf{x})d\mu(\mathbf{x})=\mu(Y),
  \end{equation}
\end{lem}
\begin{proof}
  Due to the regularity of the measure, for any positive $\epsilon$, there is an open set $U_\epsilon$ containing $Y$ that
  $\mu(U_\epsilon)\le \mu Y+\epsilon$.
  Consider the measure $ \mu_U$ defined for any $\mu$ - measurable $T$ by equality
  \begin{equation}\label{mu-restr}
    \mu_U(T)=\mu(U\cap T)
    \end{equation}
Consider any monad $\{M_n\}$ of a given partition. If its limit $M$ does not belong to $Y$, then $\mathbf{1}_Y(M)=0$ and so the distribution $\mathbf{1}_Y d\mu$ is null on this monad.
 If the limit belongs to $Y$, then it also belongs to $U$, and it follows from the openness of the latter that all elements of the monad are from some  the ultrafilter elements are contained in $U$, so $d\mu_U$ coinides with $d\mu$ on this monad.
    In both cases considered, we have the inequality $d\mu_U * \succ \mathbf{1}_Yd\mu$. Therefore,
   $d\mu_U$ majorates $ \mathbf{1}_Yd\mu$ on all monads, from where $ \int \mathbf{1}_Yd\mu\le \mu(U)\le \mu(Y)+\epsilon$.
  Due to the arbitrariness of $\epsilon$ , this implies the  inequality $ \int\mathbf{1}_Yd\mu\le\mu(Y)$.
  The opposite inequality follows from similar arguments about the complement to $Y$.
\end{proof}
Since linear combinations of indicators uniformly approximate any bounded function, the proved Lemma implies the integrability of $f(x)d\mu(x)$ type distributions for any Lebesgue-Stieltjes integrable function $f(x)$,  and the coincidence of the nonstandard (as well as standart) Leibniz integral with the corresponding Lebesgue-Stieltjes integral.

\paragraph{Newton-Leibniz formula.}
Let $X=[a,b]$ be a segment of real line. Let $\X$ be an (ultra or monotone) infinitesimal partition of $X$
by closed intervals.  For a real function $f(x)$ on $[a,b]$ and any monad $\x=\{[a_n,b_n]\}$
of the partition  is defined
the differential $df(\x)$ by the rule:
\begin{equation}\label{dif-func}
  df(\x)=\{f(b_n)-f(a_n)\}
\end{equation}

The difference between standard and nonstandard Leibniz integrals manifests itself when integrating the derivative.
The difference between the form $f'(x)$ and the ratio $df(x)/dx$  is infinitely small for any nested  sequence of segments
(see \cite{Scepin}), hence for the standard Leibniz integral  one has the equality:
\begin{equation}\label{Newton-Leibniz}
  \int_{\X}f'(x)dx=\int_{\X} df(x)=b-a
\end{equation}
But this is not the case for an arbitrary  converging to $x$ sequence of intervals, as in the case of ultra-monads.
Therefore, the validity of the Newton-Leibniz formula for a non-standard integral can be easily proved only for continuous functions. And in general, for a non-Lebesgue-integrable derivative, the form $f'(x)dx$ is not integrable by the non-standard Leibniz.

\paragraph{Infinitesimal forms}
In the "standard" version of the integral theory proposed above, we operated with sequences, whereas in the "non-standard" version with classes of their ultra-equivalence. In fact, even in the standard version, it is more accurate to operate with classes of \emph{eventually equivalent sequences}. The eventual equivalence of sequences means that they coincide everywhere except for a finite number of members. We will call the class of eventually equivalent sequences an \emph{eventual sequence}.

А monotone eventual sequence of sets $\{U_n\}$ converging to a point is called \emph{eventually infinitesimal set}.
And an ultra-sequences of sets $\{U_n\}$ is called \emph{ultra-infinitesimal set}
if it ultra-converges to a point.
A function that matches to  infinitesimal sets a  corresponding class of numerical sequences is
called the \emph{infinitesimal measure}.
Every measure $\mu$  generates  an  infinitesimal measure $d\mu$ by the usual rule
\eqref{measure-diff}.

An \emph{infinitesimal form} is any function defined on infinitesimal sets
and taking values in corresponding (eventual or ultra) equivalence classes of
numerical sequences.
The main examples of infinitesimal forms are measure differentials and
superposition of function $f\colon X\to\R$ with  passage to the limit $\mathbf{x}\to x$.
Infinitesimal forms can be added, multiplied, divided, taken by absolute value, and so on.

Infinitesimal form is called \emph{integrable} if it is integrable
over each infinitesimal partition, and all its integral coincide.

 The most common is an infinitesimal  form of the type $f(x)d\mu(\x)$.
Loebe in \cite{Loeb}  uses a different interpretation of $fd\mu$.
Assume that a point $p(U)$ is selected for each set
 $U\in\Omega$.
 In this case, the value of the function on the monad can be determined by the following rule
\begin{equation}\label{func-measure1}
  f(\{U_n\})=\{f(p(U_n))\}
\end{equation}
 It gives infinitesimally close results to the usual rule $f(\{U_n\})=f(*\lim U_n)$
 in the case of a special (so called minimal) partition considered by him.
 Вut in our case the infinitesimal closeness is provided only  for continuous functions.

The concept of a measure differential introduced above should be called a \emph{definite differential}. Whereas an \emph{indefinite differential} can be called a function of an infinitesimal set (i.e. infinitesimal form) defined by the same relation. An indefinite differential generates a definite, on a given infinitesimal partition of space. The relation between a definite and an indefinite differential is similar to the relation between a definite and an indefinite integral.


\begin{thebibliography}{9}

\bibitem{Loeb}   Loeb P. A., \emph{ А nonstandard representation of measurable spaces and $L_\infty$}, Bulletin of the
american mathematical society, Volume 77, Number 4, July 1971

\bibitem{Scepin}  Shchepin E.V. \emph{Leibniz differential and Perron-Stieltjes integral}, Journal of Mathematical Sciences, Vol. 233, No. 1, August, 2018
 157–171

\bibitem{Kaz} Mikhail G. Katz and David M. Sherry \emph{"Leibniz’s Laws of
Continuity and Homogeneity"}
 Notices of the AMS, Volume 59, Number 11, pp. 1550-1558

\end{thebibliography}
\end{document}